\documentclass[12pt,reqno]{article}

\usepackage[usenames]{color}
\usepackage{amssymb}
\usepackage{amsmath}
\usepackage{amsthm}
\usepackage{amsfonts}
\usepackage{amscd}
\usepackage{graphicx}

\usepackage[colorlinks=true,
linkcolor=webgreen,
filecolor=webbrown,
citecolor=webgreen]{hyperref}

\definecolor{webgreen}{rgb}{0,.5,0}
\definecolor{webbrown}{rgb}{.6,0,0}

\usepackage{color}
\usepackage{fullpage}
\usepackage{float}

\usepackage{graphics}
\usepackage{latexsym}

\newcommand{\braces}{\genfrac{\lbrace}{\rbrace}{0pt}{}}

\begin{document}

\theoremstyle{plain}
\newtheorem{theorem}{Theorem}
\newtheorem{corollary}[theorem]{Corollary}
\newtheorem{proposition}{Proposition}
\newtheorem{lemma}{Lemma}
\newtheorem*{example}{Examples}
\newtheorem{remark}{Remark}

\numberwithin{equation}{section}

\begin{center}

\vskip 1cm{\large\bf  
New Polynomial Identities and Some Consequences}
\vskip 1cm
\large
Kunle Adegoke \\ 
Department of Physics and Engineering Physics\\
Obafemi Awolowo University\\
220005 Ile-Ife \\
Nigeria \\
\href{mailto:adegoke00@gmail.com }{\tt adegoke00@gmail.com}\\

\end{center}

\vskip .2 in

\begin{abstract}
Using an elementary approach involving the Euler Beta function and the binomial theorem, we derive two polynomial identities; one of which is a generalization of a known polynomial identity. Two well-known combinatorial identities, namely Frisch's identity and Klamkin's identity, appear as immediate consequences of these polynomial identities. We subsequently establish several combinatorial identities, including a generalization of each of Frisch's identity and Klamkin's identity. Finally, we develop a scheme for deriving combinatorial identities associated with polynomial identities of a certain type.
\end{abstract}



\section{Introduction}
Our purpose in writing this paper is to derive the following presumably new polynomial identities:
\begin{equation}\label{poly1}
\sum_{k = 0}^n {\binom nk\binom{k + r}s^{-1}x^k}  = \sum_{k = 0}^n {\frac{(-1)^ks}{k + s}\binom nk\binom{{k + r}}{{k + s}}^{-1}x^k \left( {1 + x} \right)^{n - k} } 
\end{equation} 
and
\begin{equation}\label{poly2}
\sum_{k = 0}^n {\binom nk\binom r{k + s}^{-1} x^k}  =\sum_{k = 0}^n {\frac{(-1)^{n-k}(r + 1)}{r - k + 1}\binom nk\binom{r - k}{r - s - n}^{-1}(1-x)^{n - k} } .
\end{equation}
In identities~\eqref{poly1} and~\eqref{poly2}, $n$ is a non-negative integer and $x$ is a complex variable. Identity~\eqref{poly1} holds for complex numbers $r$ and $s$ for which $\Re(r-s+1)>0$ and $s$ is not a non-positive integer, while~\eqref{poly2} is valid for complex numbers $r$ and $s$ such that $\Re(r-n-s+1)>0$ and $s$ is not a negative integer.

Identity~\eqref{poly1} is simpler than, and yet, generalizes Identity (4.13) of Gould's book~\cite[p.47]{gould}, the latter corresponding to the case $r=s$ in~\eqref{poly1}.

At $x=-1$, identity~\eqref{poly1} reduces to Frisch's identity~\cite{frisch}, namely,
\begin{equation}\label{frisch}
\sum_{k = 0}^n {(-1)^k\binom nk\binom{k + r}s^{-1}}  = \frac s{n+s}\binom{n+r}{n+s}^{-1},
\end{equation} 
while at $x=1$, identity~\eqref{poly2} yields Klamkin's identity:
\begin{equation}\label{klamkin}
\sum_{k = 0}^n {\binom nk\binom r{k + s}^{-1}}  = \frac{{r + 1}}{{\left( {r - n + 1} \right)}}\binom{r - n}s^{-1}.
\end{equation}
In a recent paper, Gould and Quaintance~\cite{gould2} employed the well-known formula of Gauss for the hypergeometric function to give new proofs of~\eqref{frisch} and~\eqref{klamkin}. More recently, Abel~\cite{abel} used the Euler Beta function to give elementary short proofs of the identities. Our approach also uses the Beta function and is quite similar to that of Abel. For more historical facts concerning Frisch's identity and Klamkin's identity, the reader is referred to Abel~\cite{abel} and Gould and Quaintance~\cite{gould2}. In a recent paper, Adegoke and Frontczak~\cite{adegoke24} derived many harmonic and odd harmonic number identities from Frisch's identity.

Let $m$ and $n$ be non-negative integers and let $u$ be a complex number. Among other results, we will derive the following generalizations of Frisch's identity and Klamkin's identity:
\begin{equation*}
\sum_{k = 0}^n {( - 1)^k \binom{{n}}{k}\binom{{u + n - k}}{u}\binom{{k + r}}{s}^{ - 1} } = \sum_{k = 0}^n {\frac{s}{{k + s}}\binom{{n}}{k}\binom{{u}}{{n - k}}\binom{{k + r}}{{k + s}}^{ - 1} }
\end{equation*}
and
\begin{equation*}
\sum_{k = 0}^n {\binom{{n}}{k}\binom{{u + n - k}}{{n - k}}\binom{{r}}{{k + s}}^{ - 1} }  = \sum_{k = 0}^n {\frac{r+1}{{r - k + 1}}\binom{{n}}{k}\binom{{u}}{{n - k}}\binom{{r - k}}{s}^{ - 1} }.
\end{equation*}
We will also establish the following extensions:
\begin{align*}
&\sum_{k = 0}^n {( - 1)^k k^m \binom{{n}}{k}\binom{{k + r}}{s}^{ - 1} }\\
&\qquad = \sum_{k = 0}^m {\frac{s(-1)^kk!}{{n - k + s}}\binom{{n}}{k}\braces{m+n-k}n_{n-k}\binom{{n - k + r}}{{n - k + s}}^{ - 1}  } ,
\end{align*}
and
\begin{align*}
&\sum_{k = 0}^n {k^m \binom{{n}}{k}\binom{{r}}{{k + s}}^{ - 1} }\nonumber\\
&\qquad  = \sum_{k = 0}^m {\frac{{k!(r + 1) }}{{r - n + k + 1}}\binom nk\braces mk\binom{{k + r - n}}{{k + s}}^{ - 1}  } ,
\end{align*}
where $\braces mk$ are Stirling numbers of the second kind and $\braces {m+r}{k+r}_r$ are $r-$Stirling numbers of the second kind.

Finally, we will derive the following complements of Dixon's identity:
\begin{align*}
\sum_{k = 0}^{2n} {( - 1)^k k\binom{{2n}}{k}^3 }  &= ( - 1)^n n\binom{{2n}}{n}\binom{{3n}}{n},\\
\sum_{k = 0}^{2n} {( - 1)^k k^2\binom{{2n}}{k}^3 }  &= ( - 1)^n\frac{2n^2}3\binom{{2n}}{n}\binom{{3n}}{n};
\end{align*}
and establish a general formula for the evaluation of the following sum:
\begin{equation*}
\sum_{k = 0}^{2n} {( - 1)^k k^m\binom{{2n}}{k}^3 },\quad m=1,2,\ldots.
\end{equation*}
Binomial coefficients are defined, for non-negative integers $i$ and $j$, by
\begin{equation*}
\binom ij=
\begin{cases}
\dfrac{{i!}}{{j!(i - j)!}}, & \text{$i \ge j$};\\
0, & \text{$i<j$};
\end{cases}
\end{equation*}
the number of distinct sets of $j$ objects that can be chosen from $i$ distinct objects.

Generalized binomial coefficients are defined for complex numbers $r$ and $s$ by
\begin{equation*}
\binom rs= \frac{{\Gamma (r + 1)}}{{\Gamma (s + 1)\Gamma (r - s + 1)}},
\end{equation*}
where the Gamma function, $\Gamma(z)$, is defined for $\Re(z)>0$ by
\begin{equation*}
\Gamma (z) = \int_0^\infty  {e^{ - t} t^{z - 1}dt}  = \int_0^\infty  {\left( {\log (1/t)} \right)^{z - 1}dt},
\end{equation*}
and is extended to the rest of the complex plane, excluding the non-positive integers, by analytic continuation.

\section{Proof of~\eqref{poly1} and~\eqref{poly2}}
The integration formulas required for proving~\eqref{poly1} and~\eqref{poly2} are given in Lemma~\ref{integrals}.
\begin{lemma}\label{integrals}
If $r$, $k$ and $s$ are complex numbers and $x$ is a complex variable, then
\begin{align}
\int_0^1 {y^{r + k - s} \left( {1 - y} \right)^{s - 1} dy}  &=\frac1s\binom{k + r}s^{-1},\mbox{ $\Re(r+k-s+1)>0$ and $0\ne s\not\in\mathbb Z^{-}$};\label{int1a}\\
\int_0^1 {y^{r - s} \left( {1 - y} \right)^{k + s - 1} dy}  &= \frac{1}{{k + s}}\binom{{k + r}}{{k + s}}^{ - 1},\mbox{ $\Re(r-s+1)>0$ and $\Re(k + s)>0$};\label{int1b} \\
\int_0^1 {y^{k + s} \left( {1 - y} \right)^{r - k - s} dy}  &= \frac{1}{{r + 1}}\binom{{r}}{{k + s}}^{ - 1},\mbox{ $\Re(k+s+1)>0$ and $\Re(r-k-s+1)>0$}\label{int2a},
\end{align}
and
\begin{align}\label{int2b}
&\int_0^1 {y^{n - k + s} \left( {1 - y} \right)^{r - n - s} }\nonumber\\
&\qquad  = \frac{1}{{r - k + 1}}\binom{{r - k}}{{r - s - n}},\mbox{ $\Re(n-k+s+1)>0$ and $\Re(r-n-s+1)>0$}.
\end{align}
\end{lemma}
\begin{proof}
The integrals in~\eqref{int1a}--~\eqref{int2b} are immediate consequences of the Beta function, $B(r,s)$, defined, as usual, for complex numbers $r$ and $s$ such that $\Re(r)>0$ and $\Re(s)>0$, by
\begin{equation*}
B\left( {r,s} \right) = B\left( {s,r} \right)= \int_0^1 {y^{r - 1} \left( {1 - y} \right)^{s - 1} }.
\end{equation*}
With the help of the Gamma function, the integral is evaluated as
\begin{equation*}
B\left( {r,s} \right)= \frac{\Gamma(r)\Gamma(s)}{\Gamma(r + s)}=\frac{1}{s}\binom{{r + s - 1}}{s}^{-1} = \frac{1}{r}\binom{{r + s - 1}}{r}^{-1}.
\end{equation*}
Note that in obtaining~\eqref{int2a} and~\eqref{int2b}, we also used 
\begin{equation*}
\binom{{u + 1}}{{v + 1}} = \frac {u+1}{v+1}\binom uv,
\end{equation*}
an identity which we will often use without comment in this paper.
\end{proof}

\begin{theorem}
If $n$ is a non-negative integer, $x$ is a complex variable and $r$ and $s$ are complex numbers such that $\Re(r-s+1)>0$ and $s$ is not a non-positive integer, then
\begin{equation*}
\sum_{k = 0}^n {\binom nk\binom{k + r}s^{-1}x^k}  = s\sum_{k = 0}^n {\frac{(-1)^k}{k + s}\binom nk\binom{{k + r}}{{k + s}}^{-1}x^k \left( {1 + x} \right)^{n - k} } .
\end{equation*} 

\end{theorem}
\begin{proof}
Application of the binomial theorem to both sides of
\begin{equation*}
1+xy=-x(1-y)+1+x
\end{equation*}
gives
\begin{equation*}
\sum_{k = 0}^n {\binom{{n}}{k}x^ky^k }  = \sum_{k = 0}^n {\binom{{n}}{k}(-1)^kx^k \left( {1 - y} \right)^k (1  + x)^{n - k} },
\end{equation*}
which upon multiplying through by $y^{r-s}(1-y)^{s-1}$ can also be written as
\begin{align*}
&\sum_{k = 0}^n {\binom{{n}}{k}y^{r + k - s} \left( {1 - y} \right)^{s - 1}x^k }\nonumber\\
&\qquad  = \sum_{k = 0}^n {( - 1)^k\binom{{n}}{k}x^k y^{r - s} \left( {1 - y} \right)^{s + k - 1} (1 + x)^{n - k} } .
\end{align*}
Thus, term-wise integration gives
\begin{align*}
&\sum_{k = 0}^n {\binom{{n}}{k}x^k \int_0^1 {y^{r + k - s} \left( {1 - y} \right)^{s - 1} dy} }\\
&\qquad  = \sum_{k = 0}^n {( - 1)^k\binom{{n}}{k}x^k \left( {1 + x} \right)^{n - k}\int_0^1 {y^{r - s} \left( {1 - y} \right)^{s + k - 1} dy}  } ,
\end{align*}
from which~\eqref{poly1} follows by \eqref{int1a} and~\eqref{int1b}.
\end{proof}

\begin{theorem}
If $n$ is a non-negative integer, $x$ is a complex variable and $r$ and $s$ are complex numbers such that $\Re(r-n-s+1)>0$ and $s$ is not a negative integer, then
\begin{equation*}
\sum_{k = 0}^n {\binom nk\binom r{k + s}^{-1} x^k}  = \left( {r + 1} \right)\sum_{k = 0}^n {\frac{(-1)^{n-k}}{r - k + 1}\binom nk\binom{r - k}{r - s - n}^{-1}(1-x)^{n - k} } .
\end{equation*}
\end{theorem}
\begin{proof}
Raising both sides of
\begin{equation*}
xy + 1 - y = y\left( {x - 1} \right) + 1
\end{equation*}
to power $n$ and expanding via the binomial theorem gives
\begin{equation*}
\sum_{k = 0}^n {\binom{{n}}{k}y^k \left( {1 - y} \right)^{n - k} x^k }  = \sum_{k = 0}^n {( - 1)^{n - k} \binom{{n}}{k}y^{n - k} \left( {1 - x} \right)^{n - k} } 
\end{equation*}
which, after multiplying through by $y^s(1-y)^{r-n-s}$, leads to
\begin{equation*}
\sum_{k = 0}^n {\binom{{n}}{k}y^{k + s} \left( {1 - y} \right)^{r - k - s} x^k }  = \sum_{k = 0}^n {( - 1)^{n - k} \binom{{n}}{k}y^{n - k + s} \left( {1 - y} \right)^{r - n - s} \left(1 - x\right)^{n - k}} .
\end{equation*}
Performing term-wise integration, we therefore have
\begin{align*}
&\sum_{k = 0}^n {\binom{{n}}{k}x^k \int_0^1 {y^{k + s} \left( {1 - y} \right)^{r - k - s} dy} }\\
&\qquad  = \sum_{k = 0}^n {( - 1)^{n - k} \binom{{n}}{k}\left( {1 - x} \right)^{n - k} \int_0^1 {y^{n - k + s} \left( {1 - y} \right)^{r - n - s} dy} } ;
\end{align*}
and hence~\eqref{poly2} by~\eqref{int2a} and~\eqref{int2b}.
\end{proof}

\section{Combinatorial identities}
In this section we derive some combinatorial identities that are consequences of~\eqref{poly1} and~\eqref{poly2} and related identities. In particular, we will derive a generalization of each of Frisch's identity and Klamkin's identity.
\begin{proposition}
If $n$ is a non-negative integer and $r$ and $s$ are complex numbers such that $\Re(r-s+1)>0$ and $s$ is not a non-positive integer, then
\begin{equation}\label{eq.vcdwofb}
\sum_{k = 0}^n {\frac{{( - 1)^k }}{{k + s}}\binom{{n}}{k}\binom{{k + r}}{{k + s}}^{ - 1} }  = \frac{1}{s}\binom{{n + r}}{s}^{ - 1} .
\end{equation}

\end{proposition}
In particular,
\begin{equation}
\sum_{k = 0}^n {\frac{{( - 1)^k }}{{k + r}}\binom nk }  = \frac{1}{r}\binom{{n + r}}{r}^{ - 1}.
\end{equation}

\begin{proof}
By writing $1/x$ for $x$, identity~\eqref{poly1} can also be written as
\begin{equation}\label{eq.edfhvvl}
\sum_{k = 0}^n {\binom{{n}}{k}\binom{{k + r}}{s}^{ - 1} x^{n - k} }  = s\sum_{k = 0}^n {\frac{{( - 1)^k }}{{k + s}}\binom{{n}}{k}\binom{{k + r}}{{k + s}}^{ - 1} \left( {1 + x} \right)^{n - k} } ,
\end{equation}
from which~\eqref{eq.vcdwofb} is obtained by evaluating at $x=0$.
\end{proof}

\begin{remark}
Identity~\eqref{eq.vcdwofb} is the binomial transform of~\eqref{frisch}.
\end{remark}

\begin{proposition}
If $n$ is a non-negative integer and $r$ and $s$ are complex numbers such that $\Re(r-n-s+1)>0$ and $s$ is not a negative integer, then
\begin{equation}\label{eq.h2ql1u6}
\sum_{k = 0}^n {\frac{{( - 1)^k }}{{r - k + 1}}\binom{{n}}{k}\binom{{r - k}}{{r - s - n}}^{ - 1} }  = \frac{{( - 1)^n }}{{r + 1}}\binom{{r}}{s}^{ - 1}  .
\end{equation}
\end{proposition}
\begin{proof}
Set $x=0$ in~\eqref{poly2}.
\end{proof}

Obviously, Frisch-type and Klamkin-type combinatorial identities are associated with polynomial identities having the following form:
\begin{equation}\label{poly}
\sum\limits_{k = l_1 }^{l_2 } {f(k)x^{p(k)} }  = \sum\limits_{k=n_1 }^{n_2 } {g(k)\left( {1 - x} \right)^{q(k)} },
\end{equation}
where $f(k)$ and $g(k)$ are sequences of complex numbers, $p(k)$ and $q(k)$ are sequences of non-negative integers and $l_1$, $l_2$, $n_1$ and $n_2$ are non-negative integers.
\subsection{Frisch-type combinatorial identities}
\begin{theorem}\label{thm.frisch}
Let $r$ and $s$ be complex numbers such that $\Re(r+\min(p(l_1),p(l_2))-s+1)>0$ and $\Re(\min(q(n_1),q(n_2))+s)>0$; where $\min(a,b)$ picks the smaller of $a$ and $b$. If a polynomial identity has the form~\eqref{poly}, then the following combinatorial identities hold:
\begin{align}
\sum_{k = l_1 }^{l_2 } {f(k)\binom{{p(k) + r}}{s}^{ - 1}}  = s\sum_{k = n_1 }^{n_2 } {\frac{{g(k)}}{{q(k) + s}}\binom{{q(k) + r}}{{q(k) + s}}^{ - 1} }, \label{eq.yndi3o2}\\
\sum_{k = n_1 }^{n_2 } {g(k)\binom{{q(k) + r}}{s}^{ - 1} }  = s\sum_{k = l_1 }^{l_2 } {\frac{{f(k)}}{{p(k) + s}}\binom{{p(k) + r}}{{p(k) + s}}^{ - 1} } \label{eq.at6nk8c}.
\end{align}

\end{theorem}

\begin{proof}
Multiply through~\eqref{poly} by $x^{r-s}(1-x)^{s-1}$ and perform term-wise integration with respect to $x$, making use of~\eqref{int1a} and~\eqref{int1b}, thereby obtaining~\eqref{eq.yndi3o2}. Identity~\eqref{poly} can also be written as
\begin{equation*}
\sum\limits_{k = n_1 }^{n_2 } {g(k)x^{q(k)} }  = \sum\limits_{k = l_1 }^{l_2 } {f(k)\left( {1 - x} \right)^{p(k)} },
\end{equation*}
so that identities derived from~\eqref{poly} remain valid under the following transpositions:
\begin{equation*}
p(k)\leftrightarrow q(k),\quad f(k)\leftrightarrow g(k),\quad n_1\leftrightarrow l_1,\quad n_2\leftrightarrow l_2;
\end{equation*}
and hence identity~\eqref{eq.at6nk8c}.

\end{proof}

We now illustrate Theorem~\ref{thm.frisch} by deriving the Frisch-type identity associated with an identity of Simons.
\begin{lemma}[Simons~\cite{simons01}]
If $n$ is a non-negative integer and $x$ is a complex variable, then
\begin{equation}\label{simons}
\sum_{k = 0}^n {(-1)^k\binom{{n}}{k}\binom{{n + k}}{k}x^k }=\sum_{k = 0}^n {( - 1)^{n - k} \binom{{n}}{k}\binom{{n + k}}{k}\left( {1 - x} \right)^k }.
\end{equation}

\end{lemma}

\begin{proposition}
If $n$ is a non-negative integer and $r$ and $s$ are complex numbers such that $\Re(r-s+1)>0$ and $s$ is not a non-positive integer, then
\begin{equation}\label{eq.xt2yfl7}
\sum_{k = 0}^n {( - 1)^k \binom{{n}}{k}\binom{{n + k}}{k}\binom{{k + r}}{s}^{ - 1} }  = \sum_{k = 0}^n {\frac{{( - 1)^{n - k} s}}{{k + s}}\binom{{n}}{k}\binom{{n + k}}{k}\binom{{k + r}}{{k + s}}^{ - 1} } .
\end{equation}

\end{proposition}

\begin{proof}
Comparing~\eqref{poly} and~\eqref{simons}, we find
\begin{equation}\label{simons-p}
f(k) = ( - 1)^k \binom{{n}}{k}\binom{{n + k}}{k},\quad g(k) = ( - 1)^{n - k} \binom{{n}}{k}\binom{{n + k}}{k},
\end{equation}
$p(k)=k=q(k)$; and $l_1=0=n_1$ and $l_2=n=n_2$.

Using these in~\eqref{eq.yndi3o2} gives~\eqref{eq.xt2yfl7}.
\end{proof}
We can obtain a Frisch-type identity with two binomial coefficients in the denominator, directly from~\eqref{poly1}.
\begin{proposition}
If $n$ is a non-negative integer and $r$, $s$, $t$ and $u$ are complex numbers such that $\Re(r-s+1)>0$, $\Re(t-u+1)>0$ and $s$ and $u$ are not negative integers, then
\begin{align}
&\sum_{k = 0}^n {( - 1)^k \binom{{n}}{k}\binom{{k + r}}{s}^{ - 1} \binom{{k + t}}{u}^{ - 1} }\nonumber\\
&\qquad  = su\sum_{k = 0}^n {\frac{1}{{\left( {k + s} \right)\left( {n - k + u} \right)}}\binom{{n}}{k}\binom{{k + r}}{{k + s}}^{ - 1} \binom{{n + t}}{{n - k + u}}^{ - 1} } .
\end{align}
\end{proposition}

\begin{proof}
Write $-x$ for $x$ in~\eqref{poly1}, multiply through by $x^{t-u}(1-x)^{u-1}$ and integrate with respect to $x$ from $0$ to $1$.
\end{proof}

In particular,
\begin{align}
&\sum_{k = 0}^n {( - 1)^k \binom{{n}}{k}\binom{{k + r}}{s}^{ - 2} }\nonumber\\
&\qquad  = s^2\sum_{k = 0}^n {\frac{1}{{\left( {k + s} \right)\left( {n - k + s} \right)}}\binom{{n}}{k}\binom{{k + r}}{{k + s}}^{ - 1} \binom{{n + r}}{{n - k + s}}^{ - 1} } 
\end{align}
and
\begin{equation}
\sum_{k = 0}^n {( - 1)^k \binom{{n}}{k}\binom{{k + r}}{r}^{ - 2} }  = r^2\sum_{k = 0}^n {\frac{1}{{\left( {k + r} \right)\left( {n - k + r} \right)}}\binom{{n}}{k}\binom{{n + r}}{{n - k + r}}^{ - 1} } .
\end{equation}
The reader is invited to discover the combinatorial identity having two binomial coefficients in the denominator associated with~\eqref{eq.edfhvvl} by making appropriate substitutions in Theorem~\ref{thm.frisch}.
\subsubsection{A generalization of Frisch's identity}
In Theorem~\ref{thm.frisch_gen}, we derive a generalization of Frisch's identity. We require the following known polynomial identity.
\begin{lemma}[{\cite[Identity (3.18), p.24]{gould}}]
If $n$ is a non-negative integer, $u$ is a complex number and $x$ and $y$ are complex variables, then
\begin{equation}\label{gould}
\sum_{k = 0}^n {\binom{{n}}{k}\binom{{u}}{k}x^{n - k} y^k }  = \sum_{k = 0}^n {\binom{{n}}{k}\binom{{u + k}}{k}\left( {x - y} \right)^{n - k} y^k }.
\end{equation}

\end{lemma}

\begin{theorem}\label{thm.frisch_gen}
If $n$ is a non-negative integer, $u$ is a complex number and $r$ and $s$ are complex numbers such that $\Re(r-s+1)>0$ and $s$ is not a negative integer, then
\begin{align}
\sum_{k = 0}^n {\binom{{n}}{k}\binom{{u}}{{n - k}}\binom{{k + r}}{s}^{-1}}  &= \sum_{k = 0}^n {\frac{{( - 1)^k s}}{{k + s}}\binom{{n}}{k}\binom{{u + n - k}}{u}\binom{{k + r}}{{k + s}}^{ - 1} },\label{eq.ezz26ve}\\ 
\sum_{k = 0}^n {( - 1)^k \binom{{n}}{k}\binom{{u + n - k}}{u}\binom{{k + r}}{s}^{ - 1} } & = \sum_{k = 0}^n {\frac{s}{{k + s}}\binom{{n}}{k}\binom{{u}}{{n - k}}\binom{{k + r}}{{k + s}}^{ - 1} } \label{eq.n07dz4r}.
\end{align}

\end{theorem}

\begin{proof}
Setting $y=1$ in~\eqref{gould} gives
\begin{equation*}
\sum_{k = 0}^n {\binom{{n}}{k}\binom{{u}}{k}x^{n - k} }  = \sum_{k = 0}^n {( - 1)^{n - k} \binom{{n}}{k}\binom{{u + k}}{k}\left( {1 - x} \right)^{n - k} } ,
\end{equation*}
so that comparing with~\eqref{poly}, we choose
\begin{equation}\label{gould2}
f(k)=\binom nk\binom uk,\quad g(k)=(-1)^{n-k}\binom nk\binom {u+k}k,
\end{equation}
$p(k)=n-k=q(k)$, $l_1=0=n_1$ and $l_2=n=n_2$.

Substituting these in~\eqref{eq.yndi3o2} and~\eqref{eq.at6nk8c} gives
\begin{equation*}
\sum_{k = 0}^n {\binom{{n}}{k}\binom{{u}}{k}\binom{{n - k + r}}{s}^{ - 1} }  = \sum_{k = 0}^n {\frac{{( - 1)^{n - k} s}}{{n - k + s}}\binom{{n}}{k}\binom{{u + k}}{k}\binom{{n - k + r}}{{r - s}}^{ - 1} } 
\end{equation*}
and
\begin{equation*}
\sum_{k = 0}^n {( - 1)^{n - k} \binom{{n}}{k}\binom{{u + k}}{k}\binom{{n - k + r}}{s}^{ - 1} }  = \sum_{k = 0}^n {\frac{s}{{n - k + s}}\binom{{n}}{k}\binom{{u}}{k}\binom{{n - k + r}}{{r - s}}^{ - 1} } ,
\end{equation*}
which can also be written in the equivalent forms~\eqref{eq.ezz26ve} and~\eqref{eq.n07dz4r}
\end{proof}

\begin{remark}
Frisch's identity~\eqref{frisch} is obtained by setting $u=-1$ in~\eqref{eq.ezz26ve} or $u=0$ in~\eqref{eq.n07dz4r}.
\end{remark}

\subsection{Klamkin-type combinatorial identities}

\begin{theorem}\label{klamkin-type}
Let $r$ and $s$ be complex numbers such that $s$ is not a negative integer, $\Re(r-\max(q(n_1),q(n_2))-s+1)>0$ and $\Re(r-\max(p(l_1),p(l_2))-s+1)>0$; where $\max(a,b)$ picks the greater of $a$ and $b$. If a polynomial identity has the form~\eqref{poly}, then the following combinatorial identities hold:
\begin{align}
\sum_{k = n_1 }^{n_2 } {( - 1)^{q(k)} g(k)\binom{{r}}{{q(k) + s}}^{ - 1} }  &= \left( {r + 1} \right)\sum_{k = l_1 }^{l_2 } {\frac{{f(k)}}{{r - p(k) + 1}}\binom{{r - p(k)}}{s}^{ - 1} }, \label{eq.w1gtshd}\\
\sum_{k = l_1 }^{l_2 } {( - 1)^{p(k)} f(k)\binom{{r}}{{p(k) + s}}^{ - 1} }  &= \left( {r + 1} \right)\sum_{k = n_1 }^{n_2 } {\frac{{g(k)}}{{r - q(k) + 1}}\binom{{r - q(k)}}{s}^{ - 1} } .
\end{align}

\end{theorem}

\begin{proof}
Write $1/x$ for $x$ in~\eqref{poly} to obtain
\begin{equation*}
\sum\limits_{k = n_1 }^{n_2 } {( - 1)^{q(k)} g(k)x^{q(k)} \left( {1 - x} \right)^{r - q(k)} }  = \sum\limits_{k = l_1 }^{l_2 } {f(k)\left( {1 - x} \right)^{r - p(k)} },
\end{equation*}
from which, by multiplying through with $x^s(1-x)^{r-s}$, we get
\begin{equation*}
\sum\limits_{k = n_1 }^{n_2 } {( - 1)^{q(k)} g(k)x^{q(k) + s} \left( {1 - x} \right)^{r - q(k) - s} }  = \sum\limits_{k = l_1 }^{l_2 } {f(k)x^s \left( {1 - x} \right)^{r - p(k) - s} },
\end{equation*}
and hence~\eqref{eq.w1gtshd} after term-wise integration.
\end{proof}

\begin{proposition}
If $n$ is a non-negative integer and $r$, $s$, $t$ and $u$ are complex numbers such that $\Re(r-n-s)>0$, $\Re(t-n-u+1)>0$ and $s$ and $t$ are not negative integers, then
\begin{align}
&\sum_{k = 0}^n {\frac{1}{{t - k + 1}}{\binom{{n}}{k}}\binom{{t - k}}{{t - u - n}}^{ - 1} \binom{{r}}{{n - k + s}}^{ - 1} }\nonumber\\  &\qquad= \frac{{r + 1}}{{t + 1}}\sum_{k = 0}^n {\frac{1}{{r - k + 1}}{\binom{{n}}{k}}\binom{{t}}{{k + u}}^{ - 1} \binom{r - k}s^{-1}} \label{eq.jlvoaqz}
\end{align}
and
\begin{align}
&\sum_{k = 0}^n {( - 1)^k \binom{{n}}{k}\binom{{t}}{{k + u}}^{ - 1} \binom{{r}}{{k + s}}^{ - 1} }\nonumber\\  &\qquad= \left( {r + 1} \right)\left( {t + 1} \right)\sum_{k = 0}^n {\frac{{( - 1)^{n - k} \binom{{n}}{k}}}{{\left( {t - k + 1} \right)\left( {r - n + k + 1} \right)}}\binom{{t - k}}{{t - u - n}}^{-1}\binom{{r - n + k}}{s}}^{-1}\label{eq.nu5rgwf} .
\end{align}

\end{proposition}

\begin{proof}
Write~\eqref{poly2} as
\begin{equation*}
\sum_{k = 0}^n {\binom nk\binom t{k + u}^{-1} x^k}  = \left( {t + 1} \right)\sum_{k = 0}^n {\frac{(-1)^{n-k}}{t - k + 1}\binom nk\binom{t - k}{t - u - n}^{-1}(1-x)^{n - k} } .
\end{equation*}
Consider~\eqref{poly} and identify
\begin{equation*}
f(k) = \binom{{n}}{k}\binom{{t}}{{k + u}}^{ - 1} ,\quad g(k) = \frac{{( - 1)^{n - k} \left( {t + 1} \right)}}{{t - k + 1}}\binom{{n}}{k}\binom{{t - k}}{{t - u - n}}^{-1},
\end{equation*}
$p(k)=k$, $q(k)=n-k$ and $l_1=n_1=0$ and $l_2=n_2=n$.

Use these in~\eqref{eq.w1gtshd} to obtain~\eqref{eq.jlvoaqz}.

\end{proof}
\begin{proposition}
If $n$ is a non-negative integer and $r$ and $s$ are complex numbers such that $\Re(r-s-n+1)>0$ and $s$ is not a negative integer, then
\begin{align}
&\sum_{k = 0}^n {( - 1)^k \binom{{n}}{k}\binom{{r}}{{k + s}}^{ - 2} }  \nonumber\\
&\qquad= \frac{{\left( {r + 1} \right)^2 }}{{\left( {r - s - n + 1} \right)\left( {s + 1} \right)}}\sum_{k = 0}^n {( - 1)^{n - k} \binom{{n}}{k}\binom{{r - k + 1}}{{r - s - n + 1}}^{ - 1} \binom{{r - n + k + 1}}{{s + 1}}^{ - 1} } .
\end{align}
\end{proposition}
\begin{proof}
Set $t=r$ and $u=s$ in~\eqref{eq.nu5rgwf}.
\end{proof}
Our next result is a Klamkin-type identity derived from the identity of Simons~\eqref{simons}.
\begin{proposition}
Let $n$ be a non-negative integer. If $r$ and $s$ are complex numbers such that $\Re(r-n-s+1)>0$ and $s$ is not a negative integer, then
\begin{align}
&\sum_{k = 0}^n {\binom{{n}}{k}\binom{{n + k}}{k}\binom{{r}}{{k + s}}^{ - 1} }\nonumber\\
&\qquad  = \left( {r + 1} \right)( - 1)^n \sum_{k = 0}^n {\frac{(-1)^k}{{r - k + 1}}\binom nk\binom{{n + k}}{k}\binom{{r - k}}{s}^{ - 1} } .
\end{align}
\end{proposition}
\begin{proof}
Use $f(k)$, $g(k)$, etc.~given in~\eqref{simons-p} in Theorem~\ref{klamkin-type}.
\end{proof}
The reader is invited to discover the combinatorial identity having two binomial coefficients in the denominator associated with~\eqref{eq.edfhvvl} by making appropriate substitutions in Theorem~\ref{klamkin-type}.

\subsubsection{A generalization of Klamkin's identity}
We close this section by giving a generalization of Klamkin's identity.
\begin{theorem}
If $n$ is a non-negative integer, $u$ is a complex number and $r$ and $s$ are complex numbers such that $\Re(r-n-s+1)>0$ and $s$ is not a negative integer, then
\begin{align}
\sum_{k = 0}^n {\binom{{n}}{k}\binom{{u + n - k}}{{n - k}}\binom{{r}}{{k + s}}^{ - 1} }  = \left( {r + 1} \right)\sum_{k = 0}^n {\frac{1}{{r - k + 1}}\binom{{n}}{k}\binom{{u}}{{n - k}}\binom{{r - k}}{s}^{ - 1} },\label{eq.m3pmjwe} \\
\sum_{k = 0}^n {( - 1)^k \binom{{n}}{k}\binom{{u}}{{n - k}}\binom{{r}}{{k + s}}^{ - 1} }  = \left( {r + 1} \right)\sum_{k = 0}^n {\frac{{( - 1)^k }}{{r - k + 1}}\binom{{n}}{k}\binom{{u + n - k}}{{n - k}}\binom{{r - k}}{s}^{ - 1} }\label{eq.wt9v34o}. 
\end{align}
\end{theorem}
\begin{proof}
Use $f(k)$, $g(k)$, etc.~given in~\eqref{gould2} in Theorem~\ref{klamkin-type}.

\end{proof}

\begin{remark}
Klamkin's identity~\eqref{klamkin} is obtained by setting $u=0$ in~\eqref{eq.m3pmjwe}, while $u=0$ in~\eqref{eq.wt9v34o} recovers~\eqref{eq.h2ql1u6}.
\end{remark}

\section{Extensions of Frisch's identity and Klamkin's identity}

In this section we derive a closed form for the following combinatorial sums:
\begin{equation*}
\sum_{k = 0}^n {(-1)^kk^m\binom nk\binom{k + r}s^{-1}},\quad \sum_{k = 0}^n {(-1)^kk^m\binom nk\binom{r}{k + s}^{-1}},
\end{equation*}
where $m$ is a non-negative integer; thereby providing an extension for each of Frisch's identity~\eqref{frisch} and Klamkin's identity~\eqref{klamkin}, the latter corresponding to $m=0$. For this purpose we require the results stated in Lemma~\ref{m-derivatives}.
\begin{lemma}\label{m-derivatives}
If $m$, $u$ and $v$ are non-negative integers, then
\begin{equation}\label{eflzmiz}
\left. {\frac{{d^m }}{{dx^m }}\left( {{1 - e^x }^u } \right)} \right|_{x = 0}  = \sum_{p = 0}^u {( - 1)^p \binom{{u}}{p}p^m }=(-1)^uu!\braces mu
\end{equation}
and
\begin{equation}\label{eq.vompp34}
\left. {\frac{{d^m }}{{dx^m }}\left( {\left( {1 - e^x } \right)^u e^{vx} } \right)} \right|_{x = 0}  = \sum_{p = 0}^u {( - 1)^p \binom{{u}}{p}\left( {v + p} \right)^m }=(-1)^uu!\braces{m+v}{u+v}_v ,
\end{equation}
where $\braces mn$ and $\braces{m+r}{n+r}_r$ are, respectively, Stirling numbers of the second kind and $r$-Stirling numbers of the second kind.  Details about these special numbers can be found in Laissaoui and Rahmani~\cite{laissaoui17} and references therein.

\end{lemma}

\begin{proof}

Since
\begin{equation*}
\left( {1 - e^x } \right)^u e^{vx}  = \sum_{p = 0}^u {( - 1)^p \binom{{u}}{p}e^{(v + p)x} } ;
\end{equation*}
we have
\begin{equation*}
\frac{{d^m }}{{dx^m }}\left( {\left( {1 - e^x } \right)^u e^{vx} } \right) = \sum_{p = 0}^u {( - 1)^p \binom{{u}}{p}\left( {v + p} \right)^m e^{(v + p)x} };
\end{equation*}
and hence~\eqref{eq.vompp34}.

\end{proof}

We will often make use of the property
\begin{equation*}
\braces mn=0,\text{ if $n>m$.}
\end{equation*}

\subsection{An extension of Frisch's identity}

\begin{proposition}\label{prop.dfd96f9}
If $m$ and $n$ are non-negative integers and $r$ and $s$ are complex numbers such that $\Re(r-s+1)>0$ and $s$ is not a non-positive integer, then
\begin{align}
&\sum_{k = 0}^n {( - 1)^k k^m \binom{{n}}{k}\binom{{k + r}}{s}^{ - 1} } \nonumber\\
&\qquad = \sum_{k = 0}^m {\frac{{ s(-1)^kk!}}{{n - k + s}}\braces{m+n-k}n_{n-k}\binom{{n}}{k}\binom{{n - k + r}}{{n - k + s}}^{ - 1} }\label{extension} .
\end{align}
\end{proposition}
In particular,
\begin{equation}
\sum_{k = 0}^n {( - 1)^k k\binom{{n}}{k}\binom{{k + r}}{s}^{ - 1} }  = \frac{{ns}}{{n + r}}\,\frac{{s - r - 1}}{{n + s - 1}}\,\binom{{n + r - 1}}{{n + s - 1}}^{ - 1} 
\end{equation}
and
\begin{equation}
\sum_{k = 0}^n {( - 1)^k k^2 \binom{{n}}{k}\binom{{k + r}}{s}^{ - 1} }  = \frac{{ns}}{{n + r}}\,\frac{{r - s + 1}}{{r + n - 1}}\,\frac{{n\left( {r - s + 1} \right) - r}}{{n + s - 2}}\,\binom{{n + r - 2}}{{n + s - 2}}^{ - 1} ,
\end{equation}
with the special values
\begin{equation}
\sum_{k = 0}^n {( - 1)^k k\binom{{n}}{k}\binom{{k + r}}{r}^{ - 1} }  =  - \frac{{nr}}{{\left( {n + r - 1} \right)\left( {n + r} \right)}}
\end{equation}
and
\begin{equation}
\sum_{k = 0}^n {( - 1)^k k^2 \binom{{n}}{k}\binom{{k + r}}{r}^{ - 1} }  =  \frac{{nr\left( {n - r} \right)}}{{\left( {n + r} \right)\left( {n + r - 1} \right)\left( {n + r - 2} \right)}}.
\end{equation}

\begin{proof}
Write $-\exp x$ for $x$ in~\eqref{poly1} and differentiate $m$ times with respect to $x$ to obtain
\begin{align*}
&\sum_{k = 0}^n {( - 1)^k k^m \binom{{n}}{k}\binom{{k + r}}{s}^{ - 1} e^{kx} } \\
&\qquad = \sum_{k = 0}^n {\frac{s}{{k + s}}\binom{{n}}{k}\binom{{k + r}}{{k + s}}^{ - 1} \frac{{d^m }}{{dx^m }}\left(\left( {1 - e^x } \right)^{n - k} e^{kx}\right) } \\
&\qquad = \sum_{k = 0}^n {\frac{s}{{n - k + s}}\binom{{n}}{k}\binom{{n - k + r}}{{n - k + s}}^{ - 1} \frac{{d^m }}{{dx^m }}\left(\left( {1 - e^x } \right)^k e^{(n-k)x}\right) }.
\end{align*}
Evaluation at $x=0$ gives
\begin{align*}
&\sum_{k = 0}^n {( - 1)^k k^m \binom{{n}}{k}\binom{{k + r}}{s}^{ - 1} }\\
&\qquad  = \sum_{k = 0}^n {\frac{s}{{n - k + s}}\binom{{n}}{k}\binom{{n - k + r}}{{n - k + s}}^{ - 1} \left. {\frac{{d^m }}{{dx^m }}\left( {1 - e^x } \right)^k e^{(n - k)x} } \right|_{x = 0} } \\
&\qquad  = \sum_{k = 0}^m {\frac{s}{{n - k + s}}\binom{{n}}{k}\binom{{n - k + r}}{{n - k + s}}^{ - 1} \left. {\frac{{d^m }}{{dx^m }}\left( {1 - e^x } \right)^k e^{(n - k)x} } \right|_{x = 0} },
\end{align*}
and hence~\eqref{extension}, in view of Lemma~\ref{m-derivatives}.

\end{proof}

\subsection{An extension of Klamkin's identity}
\begin{proposition}
If $m$ and $n$ are non-negative integers and $r$ and $s$ are complex numbers such that $\Re(r-n-s+1)>0$ and $s$ is not a negative integer, then
\begin{align}
&\sum_{k = 0}^n {k^m \binom{{n}}{k}\binom{{r}}{{k + s}}^{ - 1} }\nonumber\\
&\qquad  = \left( {r + 1} \right)\sum_{k = 0}^m {\frac{{k!}}{{r - n + k + 1}}\binom{{n}}{k}\binom{{k + r - n}}{{k + s}}^{ - 1}\braces mk } .
\end{align}
\end{proposition}
In particular,
\begin{equation}
\sum_{k = 0}^n {k\binom{{n}}{k}\binom{{r}}{{k + s}}^{ - 1} }  = \frac{{n\left( {r + 1} \right)}}{{\left( {r - n + 2} \right)}}\binom{{r - n + 1}}{{s + 1}}^{-1},
\end{equation}
and
\begin{equation}
\sum_{k = 0}^n {k^2 \binom{{n}}{k}\binom{{r}}{{k + s}}^{ - 1} }  = \frac{{\left( {r + 1} \right)n\left( {n\left( {s + 1} \right) + r - s + 1} \right)}}{{\left( {n - r - 2} \right)\left( {n - r - 3} \right)}}\binom{{r - n + 1}}{{s + 1}}^{ - 1} .
\end{equation}
\begin{proof}
Write $\exp x$ for $x$ in~\eqref{poly2} and proceed as in the proof of Proposition~\ref{prop.dfd96f9}.
\end{proof}

\subsection{Related identities derived from polynomial identities of a certain type}

\begin{theorem}\label{thm.an0mkzw}
Let a polynomial identity have the following form:
\begin{equation}\label{eq.xtlkg2f}
\sum\limits_{k = 0}^n {f(k)x^k }  = \sum\limits_{k = 0}^n {g(k)\left( {1 - x} \right)^k },
\end{equation}
where $n$ is a non-negative integer, $x$ is a complex variable and $f(k)$ and $g(k)$ are sequences. If $m$ is a non-negative integer, then
\begin{align}
\sum_{k = 0}^n {k^m f(k)} & = \sum_{k = 0}^m {(-1)^kk!\braces mkg(k) },\label{eq.rkgu74f}\\ 
\sum_{k = 0}^n {k^m g(k)}  &= \sum_{k = 0}^m {(-1)^kk!\braces mkf(k) } \label{eq.s3sdwv2} ,\\
\sum_{k = 0}^n {k^m f(n - k)}  &= \sum_{k = 0}^m {k!\braces {n+m-k}n_{n-k}g(k) }\label{eq.aft8sl8} ,
\end{align}
and
\begin{equation}\label{eq.ik85aw8}
\sum_{k = 0}^n {k^m g(n - k)}  = \sum_{k = 0}^m {k!\braces {n+m-k}n_{n-k}f(k) } .
\end{equation}

\end{theorem}

\begin{proof}
Write $\exp x$ for $x$ in~\eqref{eq.xtlkg2f}, differentiate $m$ times with respect to $x$ and evaluate at $x=0$, using Lemma~\ref{m-derivatives}; this gives~\eqref{eq.rkgu74f}. To drive~\eqref{eq.aft8sl8}, write $1/x$ for $x$ in~\eqref{eq.xtlkg2f}, multiply through by $x^n$, replace $x$ by $\exp x$, differentiate $m$ times with respect to $x$ and evaluate at $x=0$, using Lemma~\ref{m-derivatives}. Identities~\eqref{eq.s3sdwv2} and~\eqref{eq.ik85aw8} follow from the $f(k)\leftrightarrow g(k)$ symmetry of~\eqref{eq.xtlkg2f}.
\end{proof}

\section{More combinatorial identities}
This section contains further identities based on the integration formulas in Lemma~\ref{integrals}, the evaluated derivatives in Lemma~\ref{m-derivatives} and the identities stated in Theorem~\ref{thm.an0mkzw}.
\subsection{Identities derived from the geometric progression}
\begin{proposition}
If $n$ is a non-negative integer and $r$ and $s$ are complex numbers excluding the set of non-positive integers, then
\begin{equation}\label{geometric}
\sum_{k = 0}^n {\binom{{k + r}}{s}^{ - 1} }  = \frac{s}{{s - 1}}\left(\binom{{r - 1}}{{s - 1}}^{ - 1}  - \binom{{n + r}}{{s - 1}}^{ - 1}\right) .
\end{equation}
\end{proposition}

\begin{proof}
Multiply through the sum of the geometric progression:
\begin{equation}\label{geom-sum}
\sum_{k=0}^n x^k=\frac{1 - x^{n+1}}{1 - x},
\end{equation}
by $x^{r-s}(1 - x)^{s - 1}$ to obtain
\begin{equation}
\sum\limits_{k = 0}^n {x^{k + r - s} \left( {1 - x} \right)^{s - 1} }  = x^{r - s} \left( {1 - x} \right)^{s - 2}  - x^{n + r - s + 1} \left( {1 - x} \right)^{s - 2},
\end{equation}
and hence~\eqref{geometric} by termwise integration.
\end{proof}

\begin{remark}
The special case $s=r$ of~\eqref{geometric} was proved by Rockett~\cite{rocket81}.
\end{remark}
\begin{proposition}
If $n$ is a non-negative integer, $s$ is a complex number excluding the set of negative integers and $r$ is a complex number such that $\Re(r-n-s+1)>0$, then
\begin{equation}\label{eq.p9s6a2a}
\sum_{k = 0}^n {( - 1)^k \binom{{r}}{{k + s}}^{ - 1} }  = \frac{{r + 1}}{{s + 1}}\binom{{r + 2}}{{s + 1}}^{ - 1}  + ( - 1)^n \frac{{r + 1}}{{n + s + 2}}\binom{{r + 2}}{{n + s + 2}}^{ - 1} .
\end{equation}
\end{proposition}
\begin{proof}
Write $-x$ for $x$ in~\eqref{geom-sum}, replace $x$ with $x/(1-x)$ and multiply through by $x^s(1-x)^{r-n-s}$ to obtain
\begin{equation*}
\sum_{k = 0}^n {( - 1)^k x^{k + s} \left( {1 - x} \right)^{r - k - s} }  = x^s \left( {1 - x} \right)^{r - s + 1}  + ( - 1)^n x^{n + s + 1} \left( {1 - x} \right)^{r - n - s} ,
\end{equation*}
and hence~\eqref{eq.p9s6a2a}.
\end{proof}

\subsection{An identity derived from Waring's formula}

Waring's formula is \cite[Equation (22)]{gould99}
\begin{equation}\label{waring}
\sum_{k = 0}^{\left\lfloor {n/2} \right\rfloor } {( - 1)^k \frac{n}{{n - k}}\binom {n-k}k(xy)^k (x + y)^{n - 2k} }  = x^{n}  + y^{n},
\end{equation}
and holds for a positive integer $n$ and complex variables $x$ and $y$.
\begin{proposition}
If $n$ is a non-negative integer and $r$ and $s$ are complex numbers excluding the set of negative integers and such that $\Re(r-s+1)>0$, then
\begin{equation}\label{eq.ty5mw9n}
\sum_{k = 0}^{\left\lfloor {n/2} \right\rfloor } {\frac{{( - 1)^k n}}{{\left( {n - k} \right)\left( {k + s} \right)}}\binom{{n - k}}{k}\binom{{2k + r}}{{k + s}}^{-1}}  = \frac{1}{s}\binom{{n + r}}{s}^{ - 1}  + \frac{1}{{n + s}}\binom{{n + r}}{{n + s}}^{-1}.
\end{equation}
\end{proposition}

\begin{proof}
Set $y=1-x$ in~\eqref{waring} and multiply through by $x^{r-s}(1-x)^{s-1}$ to obtain
\begin{align*}
\sum_{k = 0}^{\left\lfloor {n/2} \right\rfloor } {( - 1)^k \frac{n}{{n - k}}\binom{{n - k}}{k}x^{k + r - s} \left( {1 - x} \right)^{k + s - 1} }  = x^{n + r - s} \left( {1 - x} \right)^{s - 1}  + x^{r - s} \left( {1 - x} \right)^{n + s - 1} ,
\end{align*}
from which~\eqref{eq.ty5mw9n} follows.
\end{proof}

\subsection{More identities from an identity of Simons}

\begin{proposition}
If $m$ and $n$ are non-negative integers, then
\begin{align}
\sum_{k = 0}^n {( - 1)^{n-k} k^m \binom{{n}}{k}\binom{{k + n}}{k}}  &= \sum_{k = 0}^m {k! \binom{{n}}{k}\binom{{k + n}}{k}\braces mk },\label{eq.sn564p1} \\
\sum_{k = 0}^n {( - 1)^k k^m \binom{{n}}{k}\binom{{2n - k}}{{n - k}}}  &= \sum_{k = 0}^m {(-1)^kk!\binom{{n}}{k}\binom{{k + n}}{k}\braces{m+n-k}n_{n-k}  }\label{eq.koo68xi} .
\end{align}
\end{proposition}

\begin{proof}
Ue $f(k)$ and $g(k)$ from~\eqref{simons-p} in identities~\eqref{eq.rkgu74f} and~\eqref{eq.aft8sl8} of Theorem~\ref{thm.an0mkzw}.
\end{proof}
Explicit examples from~\eqref{eq.sn564p1} include
\begin{align}
\sum_{k = 0}^n {( - 1)^k k\binom{{n}}{k}\binom{{k + n}}{k}}  &= ( - 1)^n n\left( {n + 1} \right),\\
\sum_{k = 0}^n {( - 1)^k k^2 \binom{{n}}{k}\binom{{k + n}}{k}}  &= ( - 1)^n \frac{{n^2 \left( {n + 1} \right)^2 }}{2},
\end{align}
and
\begin{equation}
\sum_{k = 0}^n {( - 1)^k k^3 \binom{{n}}{k}\binom{{k + n}}{k}}  = ( - 1)^n \frac{{n^2 \left( {n + 1} \right)^2 \left( {n^2  + n + 1} \right)}}{6};
\end{equation}
while examples from~\eqref{eq.koo68xi} include
\begin{align}
\sum_{k = 0}^n {( - 1)^k k\binom{{n}}{k}\binom{{2n - k}}{{n - k}}}  &=  - n^2 ,\\
\sum_{k = 0}^n {( - 1)^k k^2 \binom{{n}}{k}\binom{{2n - k}}{{n - k}}}  &= \frac{{n^2 \left( {n^2  - 2n - 1} \right)}}{2},
\end{align}
and
\begin{equation}
\sum_{k = 0}^n {( - 1)^k k^3 \binom{{n}}{k}\binom{{2n - k}}{{n - k}}}  =  - \frac{{n^2 \left( {4n^2  - 6n^3  + 6n + 1 + n^4 } \right)}}{6}.
\end{equation}

\subsection{Combinatorial identities from an identity of MacMahon}
MacMahon's identity is (Gould~\cite[Identity (6.7)]{gould}, also Sun~\cite{sun12})
\begin{equation}\label{macmahon}
\sum_{k = 0}^n {( - 1)^k \binom{{n}}{k}^3 x^k }  = \sum_{k = 0}^n {( - 1)^k \binom{{n + k}}{{2k}}\binom{{2k}}{k}\binom{{n - k}}{k}x^k \left( {1 - x} \right)^{n - 2k} } .
\end{equation}

\begin{proposition}
If $n$ is a non-negative integer and $r$ and $s$ are complex numbers such that $\Re(r - s + 1)>0$ and $s$ is not a non-positive integer, then
\begin{align}
&\sum_{k = 0}^n {( - 1)^k \binom{{n}}{k}^3 \binom{{k + r}}{s}^{ - 1} }\nonumber\\
&\qquad  = \sum_{k = 0}^n {\frac{{( - 1)^k s}}{{n - 2k + s}}\binom{n + k}{2k}\binom{{2k}}{k}\binom{{n - k}}{k}\binom{{n + r - k}}{{n - 2k + s}}^{-1}} .
\end{align}

\end{proposition}

\begin{proof}
Multiply through~\eqref{macmahon} by $x^{r-s}(1-x)^{s-1}$ and integrate from $0$ to $1$.
\end{proof}

\begin{proposition}
If $n$ is a non-negative integer and $m$ is a positive integer, then
\begin{align}\label{eq.l1fpqub}
&\sum_{k = 0}^{2n} {( - 1)^k k^m \binom{{2n}}{k}^3 }\nonumber\\
&\qquad  = \sum_{k = n - m + 1}^n {( - 1)^k \binom{{2n + k}}{{2k}}\binom{{2k}}{k}\binom{{2n - k}}{k}\sum_{p = 0}^{2\left( {n - k} \right)} {( - 1)^p \binom{{2\left( {n - k} \right)}}{p}\left( {k + p} \right)^m } } .
\end{align}
\end{proposition}
In particular,
\begin{equation}
\sum_{k = 0}^{2n} {( - 1)^k k\binom{{2n}}{k}^3 }  = ( - 1)^n n\binom{{2n}}{n}\binom{{3n}}{n}
\end{equation}
and
\begin{equation}
\sum_{k = 0}^{2n} {( - 1)^k k^2\binom{{2n}}{k}^3 }  = ( - 1)^n\frac23 n^2\binom{{2n}}{n}\binom{{3n}}{n}.
\end{equation}

\begin{proof}
Write $2n$ for $n$ in~\eqref{macmahon} to obtain
\begin{equation}\label{macmahon2}
\sum_{k = 0}^{2n} {( - 1)^k \binom{{2n}}{k}^3 x^k }  = \sum_{k = 0}^{2n} {( - 1)^k \binom{{2n + k}}{{2k}}\binom{{2k}}{k}\binom{{2n - k}}{k}x^k \left( {1 - x} \right)^{2(n - k)} } ,
\end{equation}
from which~\eqref{eq.l1fpqub} follows after replacing $x$ with $\exp x$, differentiating $m$ times with respect to $x$ and evaluating at $x=0$.
\end{proof}

\begin{remark}
Setting $x=1$ in~\eqref{macmahon} gives
\begin{equation}\label{dixon}
\sum_{k = 0}^n {( - 1)^k \binom{{n}}{k}^3 }  
= \begin{cases}
 ( - 1)^{n/2} \binom{{n}}{{n/2}}\binom{{3n/2}}{n},&\text{if $n$ is even;} \\ 
 0,&\text{if $n$ is odd;}\\ 
 \end{cases} 
\end{equation}
which subsumes
\begin{equation}
\sum_{k = 0}^{2n} {( - 1)^k \binom{{2n}}{k}^3 }  = ( - 1)^n \binom{{2n}}{n}\binom{{3n}}{n}.
\end{equation}
Identity~\eqref{dixon} is Dixon's identity~\cite{dixon03} in its original form~\cite{ward91}.
\end{remark}

\bigskip
\hrule
\bigskip


\begin{thebibliography}{99}

\bibitem{abel} U. Abel, A short proof of the binomial identities of Frisch and Klamkin, J. Integer Seq. 23 (2020), Article 20.7.1.

\bibitem{adegoke24} K. Adegoke and R. Frontczak, Some notes on an identity of Frisch, \emph{Open J. Math. Sci.} {\bf 8} (2024), 216--226.

\bibitem{dixon03} A. C. Dixon, Summation of a certain series, \emph{Proc. London Math. Soc.} {\bf 35} (1903), 284--289.

\bibitem{frisch} R. Frisch, Sur les semi-invariants et moments employ\'{e}s dans l\'{e}tude des distributions statistiques, \emph{Skrifter utgitt av Det Norske Videnskaps-Akademi i Oslo, II, Historisk-Filosofisk Klasse}, {\bf 3} (1926), 1--87.

\bibitem{gould} H. W. Gould, \emph{Combinatorial Identities}, Morgantown, West Virginia, 1972.

\bibitem{gould99} H. W. Gould, The Girard-Waring power sum formulas for symmetric functions and Fibonacci sequences, \emph{Fibonacci Quart.} {\bf 37} (1999), 135--140.

\bibitem{gould2} H. W. Gould and J. Quaintance, On the binomial identities of Frisch and Klamkin,
J. Integer Seq. 19 (2016), Article 16.7.7.

\bibitem{laissaoui17} D. Laissaoui and M. Rahmani, An explicit formula for sums of powers of integers in terms of Stirling numbers, \emph{ J. Integer Seq.} {\bf 20} (2017), Article 17.4.8.

\bibitem{rocket81} A. M. Rockett, Sums of the inverses of binomial coefficients, \emph{Fibonacci Quart.} {\bf 19} (1981), 433--437.

\bibitem{simons01} S. Simons, A curious identity, \emph{Math. Gaz.} {\bf 85} (2001), 296--298.

\bibitem{sun12} Z. Sun, On sums involving products of three binomial coefficients, \emph{Acta Arith.} {\bf 156} (2012), 123--141.

\bibitem{ward91} J. Ward, 100 years of Dixon's identity, \emph{IMS Bulletin} {\bf 27} (1991), 46--54.

\end{thebibliography}
\end{document}